\documentclass[12pt]{amsart}
\usepackage{amsmath,amsthm,latexsym,amscd,amsbsy,amssymb,url}
\usepackage[all]{xypic}
 \relax

\chardef\bslash=`\\ 

\makeatletter
\def\verbatim{\interlinepenalty\@M \@verbatim
  \leftskip\@totalleftmargin\advance\leftskip2pc
  \frenchspacing\@vobeyspaces \@xverbatim}
\makeatother
\newtheorem{thm}{Theorem}[section]

\newtheorem{lem}[thm]{Lemma}

\newtheorem{defin}[thm]{Definition}

\newtheorem{que}[thm]{Question}

\begin{document}


\title
{Injections into Function Spaces over Compacta}
\author{Raushan ~Z. ~Buzyakova}
\address{Department of Mathematics and Statistics,
The University of North Carolina at Greensboro,
Greensboro, NC, 27402, USA}
\email{Raushan\_Buzyakova@yahoo.com}


\begin{abstract}{
We study the topology of $X$ given that $C_p(X)$ injects into
$C_p(Y)$, where $Y$ is compact. We first show that if $C_p$ over a GO-space injects into $C_p$
over a compactum, then the Dedekind remainder of the GO-space is hereditarily paracompact. Also, for each ordinal $\tau$ of
uncountable cofinality, we  construct
a continuous bijection of $C_p(\tau, \{0,1\})$ onto a subgroup of $C_p(\tau+1, \{0,1\})$, which is in addition a group
isomorphism.
}
\end{abstract}

\maketitle
\markboth{Raushan Z.Buzyakova}{Injections into Function Spaces over Compacta}
{ }

\section{Introduction}\label{S:intro}

\par\bigskip
In this paper we continue exploring  connections between $X$ and $Y$ given that $C_p(X)$ admits a continuous injection into $C_p(Y)$.
We first observe (Theorem \ref{thm:maincriterion}) that if a GO-space $X$ (a subspace of a linearly ordered topological space) admits a continuous injection
into $C_p(Y)$, where $Y$ is compact, then its Dedekind remainder is hereditarily paracompact. In other words, the Dedekind remainder of $X$ does not contain a subspace homeomorphic to a stationary subspace
of an uncountable regular ordinal. This observation has the same flavor as an earlier result of the author 
\cite[Theorem 2.6]{BUZ3}, that if $\tau$ is an ordinal and $X$ is a subspace of an ordinal
such that $C_p(X, \{0,1\})$ admits a continuous injection into $C_p(\tau, \{0,1\})$ then $\overline X\setminus X$ is hereditarily
paracompact. The proof of this earlier statement is rather technical and therefore, it is natural  
to ask if one can derive the earlier statement from our new one.
Clearly, neither is  a generalization of the other. To make our new statement usable for derivation of  the earlier one we prove that
$C_p(\tau, \{0,1\})$ admits a continuous isomorphism onto a subgroup of $C_p(\tau +1, \{0,1\})$ for any ordinal $\tau$ of uncountable cofinality (Theorem \ref{thm:isomorphism}). Note
that there is no  continuous surjection of $C_p(\omega_1, \{0,1\})$ onto   $C_p(\omega_1+1, \{0,1\} )$ since the former is Lindel$\ddot {\rm o}$f
and the latter is not. Also, $C_p(\omega_1, \{0,1\})$ is not homeomorphic to any subspace of  $C_p(\omega_1+1)$ since the latter has countable
tightness and the former does not. Given these observations, the map we construct, even though quite natural, may  seem    unexpected. 

In \cite[Theorem 2.9]{BUZ2}, the author showed that if $M$ is a metric space and $A\subset \omega_1$ then the existence of an injection of 
$C_p(A, M^\omega)$ into $C_p(\omega_1, M^\omega)$ is equivalent to the existence of an embedding of $C_p(A, M^\omega)$
into $C_p(\omega_1, M^\omega)$. The injective map of $C_p(\omega_1,\{0,1\})$ into $C_p(\omega_1+1, \{0,1\})$ that we construct 
shows that this  statement cannot be extended beyond $\omega_1$. Namely, it is no longer true even for $\omega_1+1$.

In notation and terminology of general topological nature we will follow \cite{ENG}. For basic facts on $C_p(X)$
we refer to \cite{ARH}. For basic facts about ordinals, we refer to \cite{KUN}. Ordinals
are endowed with the topology of linear order and their subsets with the subspace topology.

\par\bigskip
\section{Study}\label{S:study}

\par\bigskip
For convenience we next give a description of the classical construction, Dedekind completion.
Even though it can be found in many classical textbooks, we copy it from the author's earlier work \cite{BUZ1}
since we will use it together with a statement proved in that work.
\par\bigskip\noindent
{\bf Dedekind Completion.} 
An ordered pair $\langle A, B\rangle$ of 
disjoint closed subsets
of $L$ is called a {\it Dedekind section} if
$A\cup B = L$, $\max A$ or $\min B$ does not
exists, and $A$ is to the left of $B$.
A pair $\langle L, \emptyset\rangle$ ($\langle \emptyset, L\rangle$)
is also a Dedekind section if $\max L$ ($\min L$) does not exist. 
The {\it Dedekind completion} of $L$, 
denoted by $cL$, is constructed as follows.
The set $cL$ is the union of $L$ and the set of all Dedekind
sections of $L$. The order on $cL$ is natural.
The order on elements of $L$ is not changed. 
If $x\in L$ and $y=\langle A, B\rangle\in cL\setminus L$
then $x$ is less (greater) than $y$ if $x\in A$ ($B$).
If $x=\langle A_1,B_1\rangle$ and $y=\langle A_2,B_2\rangle$ 
are  elements of  $cL\setminus L$, then $x$ is less than $y$ if 
$A_1$ is a proper subset of $A_2$.

\par\bigskip
The mentioned statement of our interest follows and will be later used to prove one of our results.

\par\bigskip\noindent
\begin{thm}\label{thm:remaindertoCp} {\rm (\cite[Corollary 3.5]{BUZ1}))}
Let $L$ be a GO-space. 
Then 
the Dedekind remainder $cL\setminus L$
of $L$ is homeomorphic to a closed subspace
of $C_p(L,\{0,1\})$.
\end{thm}

\par\bigskip
Before we proceed, let us remind two concepts and a known theorem of $C_p$-theory. Given a space $X$, by $e(X)$ we denote the supremum of cardinalities
of closed discrete subspaces of $X$. By $l(X)$ we denote the smallest cardinal number such that every open cover of $X$
contains a  subcover of size at most $l(X)$. It is a well-known theorem of Baturov \cite{BAT} that $e(Z)=l(Z)$ for any subspace $Z$ of $C_p(X)$, where
$X$ is a Lindel${\rm \ddot o}$f $\Sigma$-space (a complete proof can also be found in \cite[Theorem III.6.1]{ARH}). 
We will use the fact that every Lindel${\rm \ddot o}$f locally compact space is a Lindel${\rm \ddot o}$f $\Sigma$-space 
(see \cite{NAG}). The definition of 
a Lindel${\rm \ddot o}$f $\Sigma$-space is irrelevant for our discussion, and is therefore, omitted.
\par\bigskip\noindent
\begin{lem}\label{lem:onetooneimageofstationary}
Let $S$ be a stationary subset of a regular uncountable ordinal $\kappa$. Suppose that $f:S\to X$ is a continuous
injection. Then there exists a subspace $Z\subset X$ such that $e(Z)\not = l(Z)$.
\end{lem}
\begin{proof}
First, observe that the extent of $S$ is less than $\kappa$. Indeed, let $D$ be a $\kappa$-sized subset of
$S$ which is discrete in itself. Denote by $D'$ the derived set of $D$ in $\kappa$. That is, $D' = cl_{\kappa}(D)\setminus D$.
The set $D'$ is a closed unbounded subset of $\kappa$. Since $S$ is a stationary subset of $\kappa$,
$S$ meets $D'$. That is, $D$ is not closed in $S$. Therefore, any closed discrete subset of $S$ has cardinality less than $\kappa$.  Since $\kappa$ is uncountable and regular, we conclude that $e(S)<\kappa$. 

Next, observe that $S\setminus \{s\}$  is also stationary, and therefore, has extent strictly less than $\kappa$
for any $s\in S$. By continuity of $f$, the inequalities  $e(f(S))<\kappa$ and $e(f(S)\setminus \{x\})<\kappa$ hold for any $x\in X$.
Thus, to prove our lemma it suffices to show
that either $f(S)$ or $f(S)\setminus \{x\}$ for some $x\in X$ has Lindel\"of number at least $\kappa$.

Since $S$ has only one complete accumulation point in $\beta S$, we may assume that this point is $\kappa$.
Let $\tilde f: \beta S\to \beta X$ be the continuous extension to the \v Cech-Stone compactifications.
Since $\kappa$ is the only complete accumulation point of $S$ and $\tilde f$ is continuous, we conclude that 
$\tilde f(S)\setminus U$ is of cardinality strictly less than $\kappa$ for any neighborhood $U$ of $\tilde f(\kappa)$.
Since $|f(S)| = \kappa$, we conclude that the pseudocharacter of $\tilde f(\kappa)$ in $f(S)\cup \{ \tilde f(\kappa)\}$
is at least $\kappa$. Thus, $Z = f(S) \setminus \{\tilde f(\kappa)\}$ has Lindel\"of number at least kappa and extent strictly 
less than $\kappa$. Hence $Z$ is as desired.
\end{proof}

\par\bigskip\noindent
\begin{lem}\label{lem:nostationaryinCpcompact}
No stationary subset of an uncountable regular ordinal
admits a continuous injection into
$C_p(X)$, where $X$ is a Lindel${\ddot o}$f $\Sigma$-space.
\end{lem}
\begin{proof}
By Baturov's theorem $l(Z) = e(Z)$ for every $Z\subset C_p(X)$.
Now apply Lemma \ref{lem:onetooneimageofstationary}.
\end{proof}

\par\bigskip\noindent
\begin{thm}\label{thm:maincriterion}
Let $L$ be a GO-space and $X$ a Lindel${\ddot o}$f $\Sigma$-space. If $C_p(L,\{0,1\})$ admits a continuous injection
into $C_p(X)$ then
\begin{enumerate}
	\item $cL\setminus L$ does not contain a subspace homeomorphic to a stationary subset of an uncountable
	regular ordinal; and
	\item $cL\setminus L$ is hereditarily paracompact.
\end{enumerate}
\end{thm}
\begin{proof}
Theorem \ref{thm:remaindertoCp} and Lemma \ref{lem:nostationaryinCpcompact} imply (1).

To show (2), we will use a classical  theorem of Engelking and Lutzer \cite{LUT} stating
that a GO-space is paracompact iff it contains a closed subspace homeomorphic to a stationary
subset of a regular uncountable ordinal. This theorem implies that a GO-space is hereditarily paracompact iff
it does not contains a subspace homeomorphic to a stationary subset of an uncountable regular ordinal. This criterion and (1)
imply (2).
\end{proof}

\par\bigskip
Observe that if $X$ is a GO-space and has uncountable extent, then $C_p(X)$ does not admit a continuous injection
into $C_p$ over a compactum simply because $C_p(X)$ has a subspace homeomorphic to $\{0,1\}^{\omega_1}$
while $C_p$ over a compactum cannot have such a subspace. This observation and Theorem \ref{thm:maincriterion} set quite strong requirements
on the topology of a GO-space $X$ whose function space admits a continuous injection into the function space over a compactum.

As we mentioned in the introduction, Theorem \ref{thm:maincriterion} is similar to a particular case of an earlier result of the author
that if $C_p(X,\{0,1\})$ continuously injects into $C_p(\tau,\{0,1\})$, where $X$ is a subspace of an ordinal and $\tau$ is some (other) ordinal,
then $\overline X\setminus X$ is hereditarily paracompact. Since $\tau$ need not be isolated, this statement does not
follow from Theorem \ref{thm:maincriterion}. To make the desired reduction for $\tau$ of uncountable cofinality we next  construct a continuous bijection of $C_p(\tau,\{0,1\})$
onto a subgroup of $C_p(\tau+1, \{0,1\})$, which is, in addition, a group isomorphism. We start with the following definition.
\par\bigskip\noindent
\begin{defin}\label{defin:fdeterminingsequence} Let $\tau$ be an ordinal of uncountable cofinality.
For each $f\in C_p(\tau, \{0,1\})$ we say that
$\langle i, \langle i_1,...,i_n\rangle\rangle$ is an $f$-determining sequence if the following conditions are met:
\begin{enumerate}
	\item $i_1=0$,
	\item $f(0)=i$,
	\item If $0<k\leq n$, then  $i_k = \min\{\alpha<\tau: i_{k-1} <\tau, f(\alpha)\not = f(i_{k-1})\}$,
	\item $f$ is constant on $[i_n,\tau)$.
\end{enumerate}
\end{defin}
\par\bigskip
Note that due to continuity of $f$, the ordinals $i_1,...,i_n$ are isolated.
\par\bigskip\noindent
\begin{lem}\label{lem:existenceuniquenessoffsequence}
Let $\tau$ be an ordinal of uncountable cofinality and $f\in C_p(\tau, \{0,1\})$.
Then an $f$-determining sequence exists and  is unique.
\end{lem}
\begin{proof}
Since $\tau$ has uncountable cofinality, there exists the smallest $\alpha$ such
that $f$ is constant on  $[\alpha,\tau)$. Since $[0,\alpha]$ is a zero-dimensional compactum, there exists a finite partition of $[0,\alpha]$
by convex sets on which $f$ is constant. The left-endpoints of the partition
serve as $i_1,...,i_{n-1}$ and $\alpha$ as $i_n$.
\end{proof}
\par\bigskip\noindent
We next define a correspondence from $C_p(\tau, \{0,1\})$ to $C_p(\tau+1, \{0,1\})$ that will be shown
to be a desired  map.

\par\bigskip\noindent
\begin{defin}\label{defin:injection}
Let $\tau$ be an ordinal of uncountable cofinality, $f\in C_p(\tau, \{0,1\})$, and 
$\langle i, \langle i_1,...,i_n\rangle\rangle$ the $f$-determining sequence.
Then $\phi(f):\tau+1\to \{0,1\}$ is defined as follows:
\begin{description}
	\item[\rm If $i=0$, then put]
	\[ \phi(f)(x)=   \left\{
              		\begin{array}{lll}
                   0 & if & x\not = i_2,...,i_n \\
                   1 & if & x=i_2,...,i_n
             		  \end{array}
            			\right. 
\]
	\item[\rm If $i=1$, then put]
	\[ \phi(f)(x)=   \left\{
              		\begin{array}{lll}
                   0 & if & x\not = i_1,...,i_n \\
                   1 & if & x=i_1,...,i_n
             		  \end{array}
            			\right. 
\]
\end{description}
\end{defin}
\par\bigskip
Note that $\phi(f)$ is a continuous function from $\tau+1$ to $\{0,1\}$ because $\{i_1,...,i_n\}$ in the $f$-determining sequence are isolated ordinals. By Lemma \ref{lem:existenceuniquenessoffsequence},
$\phi(f)$ is well defined for each $f$. Therefore, $\phi$ is a well-defined map from $C_p(\tau, \{0,1\})$ into $C_p(\tau+1, \{0,1\})$.
\par\bigskip\noindent
\begin{lem}\label{lem:CptautoCptauplus1} Let $\tau$ be an ordinal of uncountable cofinality. Then
$\phi:C_p(\tau, \{0,1\})\to C_p(\tau+1, \{0,1\})$ is a continuous injection.
\end{lem}
\begin{proof} The conclusion follows from the next two claims.
\par\bigskip\noindent
{\it Claim 1.} {\it
The map $\phi$ is  one-to-one. 
}
\par\smallskip\noindent
To show that $\phi$ is one-to-one, fix distinct $f$ and $g$ in  $C_p(\tau, \{0,1\})$.
Let $\langle i_f,\langle i_1^f,...,i_n^f\rangle\rangle$
and $\langle i_g,\langle i_1^g,...,i_m^g\rangle\rangle$ be the $f$ and $g$-determining sequences.
If $\langle i_1^f,...,i_n^f\rangle=\langle i_1^g,...,i_m^g\rangle$, then $f$ and $g$ are constant on the same
clopen intervals. In this case, $f\not = g$ implies that  $f(0)\not = g(0)$. That is, $i_f\not = g_f$. We may assume that $i_f=0$.
By the definition of $\phi$, we have $\phi(f)(0)=0$ and $\phi(g)(0)=1$.
Now assume that $\langle i_1^f,...,i_n^f\rangle\not = \langle i_1^g,...,i_m^g\rangle$. Since the elements of each sequence are in
increasing order, there exists
an element in one sequence which is not an element of the other. Since $i_0^f=i_0^g=0$, we may assume that
that $i_3^f$ is such an element. By the definition of $\phi$, we have $\phi(f)(i_3^f)=1$ and $\phi(g)(i_3^f)=0$.

\par\bigskip\noindent
{\it Claim 2.} {\it $\phi$ is continuous.}
\par\medskip\noindent
To prove the claim, fix an open $U$ in $C_p(\tau+1,\{0,1\})$.
We need to show that $\phi^{-1}(U)$ is open. We can assume that $U$ is an element of the standard
subbase. That is, there exist $x\in \tau+1$ and $i_x\in \{0,1\}$ such that
$U = \{g\in C_p(\tau+1,\{0,1\}):g(x)=i_x\}$. We will proceed  by exhausting all possibilities on the values of $x$ and $i_x$.
\begin{description}
	\item[\rm Case ($x=\tau, i_x=0$)]	In this case $U$ contains all functions that are eventually $0$. Since $\phi(f)$ is eventually $0$ for every
	$f\in C_p(\tau, \{0,1\})$, we conclude that $\phi^{-1}(U) = C_p(\tau, \{0,1\})$.
	\item[\rm Case ($x=\tau, i_x=1$)] By the argument of the previous case, $\phi^{-1}(U) = \emptyset$.
	\item[\rm Case ($x=0, i_x=0$)]	Since $\phi(f)(0)=0$ if and only if $f(0) = 0$, we conclude that
	$\phi^{-1}(U) = \{f: f(0)=0\}$.
	\item[\rm Case ($x=0, i_x=1$)] By the argument of the previous case, $\phi^{-1}(U) = \{f: f(0)=1\}$.
	\item[\rm Case ($x\ is\ limit,\ x<\tau, i_x=0$)]	Due to continuity, $f\in C_p(\tau, \{0,1\})$ cannot change the value
	at a limit ordinal. Applying the definition of $\phi(f)$, we have $\phi(f)(x) = 0$.  Therefore,  $\phi^{-1}(U) = C_p(\tau, \{0,1\})$.
	\item[\rm Case ($x\ is\ limit,\ x<\tau, i_x=1$)] By the argument of the previous case, $\phi^{-1}(U) = \emptyset$.
	\item[\rm Case ($x\ is\ isolated,\ 0<x<\tau, i_x=0$)]	If $f\in\phi^{-1}(U)$, then $f$ does not change its value at $x$. This means
	that $f(x)=f(x-1)$. Therefore, $\phi^{-1}(U) = \{f: f(x)=f(x-1)\}$.
	\item[\rm Case ($x\ is\ isolated,\ 0<x<\tau, i_x=1$)] By the argument of the previous case, $\phi^{-1}(U) = \{f: f(x)\not =f(x-1)\}$.
\end{description}
\end{proof}
\par\bigskip\noindent
\begin{lem}\label{lem:phiisomorphism}
$\phi$ is an isomorphism of $C_p(\tau, \{0,1\})$ with its image.
\end{lem}
\begin{proof}
Fix arbitrary $f,g\in C_p(\tau, \{0,1\})$.
Since $\phi$ is one-to-one, we only need to show that 
$\phi(f+g)(x) = (\phi(f)+\phi(g))(x)$ for each $x\in \tau+1$. This is equivalent to showing
the following equality for each $x\in \tau+1$
\par\smallskip\noindent
(*) $\phi(f+g)(x) = \phi(f)(x)+\phi(g)(x).$
\par\smallskip\noindent
We will prove this equality inductively on the value of $x$. For this let $\langle i_f,\langle i_1^f,...,i_n^f\rangle\rangle$
and $\langle i_g,\langle i_1^g,...,i_m^g\rangle\rangle$ be the $f$ and $g$-determining sequences.
\par\bigskip\noindent
{\it Step $x=0$.} 
\begin{description}
	\item[\rm Case ($f(0)\not=g(0)$) ] Then $\{f(0), g(0)\} = \{0,1\}$. Hence, $(f+g)(0) = f(0)+g(0) = 1$.
	By the definition of $\phi$, we have $\phi (f+g)(0) = 1$. Thus, the left side of (*) is $1$.
	
	Since $\{f(0), g(0)\} = \{0,1\}$ we obtain that $\{\phi(f)(0), \phi(g)(0)\} = \{0,1\}$.
	Therefore, $\phi(f)(0)+\phi(g)(0) = 1$, that is, the right side of (*) is $1$ as well.
	\item[\rm Case ($f(0)=g(0)$) ] Then $(f+g)(0)=f(0)+g(0)=0$. Therefore, the left side of (*) is $\phi(f+g)(0)=0$.
	
	Since $f(0) = g(0)$, we conclude that $\phi(f)(0)=\phi(g)(0)$. Therefore, the right side of (*) is $\phi(f)(0)+\phi(g)(0)=0$.
\end{description}
\par\bigskip\noindent
{\it Assumption.} Assume that (*) holds for all $x\in [0,\alpha)$, where  $\alpha\leq \tau$.
\par\bigskip\noindent
{\it Step $x=\alpha>0$.} Let $k\in \{1,...,n\}$ be the largest such that $\alpha\geq i_k^f$. Let $l\in \{1,...,m\}$
be the largest such that $\alpha\geq i_l^g$.
\begin{description}
	\item[\rm Case ($\alpha=i_k^f=i^g_l$)] Then $f$ and $g$ change values at $\alpha$ and hence $\alpha$ is isolated.
	Therefore, we have $f(\alpha)\not = f(\alpha-1)$ and $g(\alpha)\not = g(\alpha-1)$. Then the possibilities for
	$\langle f(\alpha-1), g(\alpha-1), f(\alpha), g(\alpha)\rangle$ are
	$\langle 1,1,0,0\rangle$, $\langle 0,0,1,1 \rangle$, $\langle 1,0, 0,1\rangle$, and $\langle 0,1,1,0 \rangle$.
	In the first two of  cases, we have $(f+g)(\alpha-1)=(f+g)(\alpha)=0$. In the other two cases, we have
	$(f+g)(\alpha-1)=(f+g)(\alpha)=1$. In all cases we have $(f+g)(\alpha-1)=(f+g)(\alpha)$. That is,
	$(f+g)$ does not change its value at $\alpha$. Therefore, the left side of (*) is  $\phi(f+g)(\alpha) = 0$.
	
	To evaluate the right side of (*), observe that the case's condition implies that
	$\phi(f)(\alpha)=\phi(g)(\alpha)=1$. Therefore, the right side of (*) is $0$ too.	
	
	\item[\rm Case ($\alpha=i_k^f\ and\ i_k^f\not =i^g_l$)] Then $f$ changes its value at $\alpha$ while $g$ does not.
	Since $f$ changes its value at $\alpha$, $\alpha$ is isolated. The possibilities for $\langle f(\alpha-1), f(\alpha)\rangle$
	are $\langle 0,1\rangle$ and $\langle 1,0\rangle$. The possibilities for $\langle g(\alpha-1), g(\alpha)\rangle$
	are $\langle 0,0\rangle$ and $\langle 1,1\rangle$. Therefore,
	the possibilities for $\langle f(\alpha-1)+g(\alpha-1), f(\alpha)+g(\alpha)\rangle$
	are $\langle 0,1\rangle$ and $\langle 1,0\rangle$. That is, $(f+g)$ changes its value at $\alpha$. Therefore,
	the left side of (*) is $\phi(f+g)(\alpha)= 1$.
	
	To evaluate the right side of (*), observe  that the case's conditions imply  that $f$ changes its value at $\alpha$ while $g$ does not.
	Therefore, $\phi(f)(\alpha) =1$ and $\phi(g)(\alpha)=0$. Therefore, the right side of (*) is
	$\phi(f)(\alpha)+\phi(g)(\alpha)=1$.
	\item[\rm Case ($\alpha=i_l^g\ and\ i_k^f\not =i^g_l$)] This case is analogous to the previous case.
	\item[\rm Case ($\alpha\not=i_k^f\ and\ \alpha\not =i^g_l$)] In this case, there exists $\beta<\alpha$ such that
	$f$ and $g$ are constant on $[\beta, \alpha]$. Then $f+g$ is constant on $[\beta, \alpha]$.
	Therefore, the left side of (*) is $\phi(f+g)(\alpha) = 0$.
	
	Let us evaluate the right side of (*). The case's conditions imply that $\phi(f)(\alpha)=0$
	and $\phi(g)(\alpha)=0$. Therefore, the right side of (*) is $\phi(f)(\alpha)+\phi(g)(\alpha)=0$. 
\end{description}
\par\bigskip\noindent
\end{proof}
\par\bigskip
We can sumarize statements \ref{lem:existenceuniquenessoffsequence}-\ref{lem:phiisomorphism}
as follows:
\par\bigskip\noindent
\begin{thm}\label{thm:isomorphism}
Let $\tau$ be an ordinal of uncontable cofinality. Then there exists a continuous one-to-one map of $C_p(\tau, \{0,1\})$
onto a subgroup of $C_p(\tau + 1, \{0,1\})$, which is, a group isomorphism.
\end{thm}

\par\bigskip
Note that $C_p(\omega,\{0,1\})$ does not admit a continuous injection into $C_p(\omega+1, \{0,1\})$ since the latter
is countable while the former is uncountable. Therefore, the condition on cofinality of $\tau$ in our construction of $\phi$
is important.
\par\bigskip
Our results can be used to derive some earlier results of the author. Namely, in
\cite[Theorem 2.6]{BUZ3}, the author proved that  if $M$ is a metric space with at least two elements, $\tau$ is an ordinal, and $X$ is a subspace of an ordinal
such that $C_p(X, M)$ admits a continuous injection into $C_p(\tau, M)$, then $\overline X\setminus X$ is hereditarily
paracompact. The results of this paper can be used to derive the mentioned earlier result for the case when $M = \{0,1\}$. Indeed,
Let $X$ be  a subspace of an ordinal and let $C_p(X, \{0,1\})$
admit a continuous injection into $C_p(\tau, \{0,1\})$ for some ordinal $\tau$.
Let $\kappa$ be the smallest ordinal number such that $X\subset \kappa$. Clearly, $Cl_{\kappa+1}(X)$  is the Dedekind completion
of $X$. If $\tau$ is of countable cofinality, then $\tau$ is locally compact or compact, that is, a Lindel${\rm \ddot o}$f $\Sigma$-space. By Theorem \ref{thm:maincriterion}, the Dedekind remainder of $X$ is hereditarily paracompact. If $\tau$ has uncountable cofinality, then
$C_p(\tau, \{0,1\})$ admits a continuous injection into $C_p(\tau+1, \{0,1\})$,  and therefore, $C_p(X, \{0,1\})$ admits a continuous
injection into $C_p(\tau+1, \{0,1\})$. Now apply Theorem \ref{thm:maincriterion} to conclude that $Cl_{\kappa+1}(X)\setminus X$ is hereditarily
paracompact.
Of course it would be nice if were able to derive the most general version of the earlier result using our new approach. But for this, we need a
positive answer to the following question.

\par\bigskip\noindent
\begin{que}
Let $\tau$ be an ordinal of uncountable cofinality and $M$ a metric space containing at least two points.
Is it true that $C_p(\tau, M)$ admits a continuous injection into $C_p(\tau+1, M)$?
\end{que}
\par\bigskip\noindent
Next are  natural questions prompted by the properties of our $\phi$ defined in Definition \ref{defin:injection}.
\par\bigskip\noindent
\begin{que}
Let $X$ be a countably compact locally compact space. Is it true that $C_p(X)$ admits a continuous injection
into $C_p(Y)$ for some compactum $Y$.
\end{que}
\par\bigskip\noindent
\begin{que}
Let $\tau$ be an ordinal of uncountable cofinality and $G$ a topological group.
Is it true that $C_p(\tau, G)$ admits a continuous isomorphism onto a subgroup of $C_p(\tau +1, G)$?
\end{que}
\par\bigskip
In \cite{BUZ2}, the author showed that for a subspace $A$ of $\omega_1$ and a non-trivial metric space $M$, the existence of an embedding
of $C_p(A, M^\omega)$ into $C_p(\omega_1, M^\omega)$ is equivalent to the existence of an injection of $C_p(A, M^\omega)$ into 
$C_p(\omega_1, M^\omega)$.
The results of this paper show that this criterion cannot be extended beyond $\omega_1$. Indeed, by Lemma \ref{lem:CptautoCptauplus1}, 
$C_p(\omega_1, \{0,1\})$ admits a continuous injection into $C_p(\omega_1+1, \{0,1\})$. Then $C_p (\omega_1, \{0,1\})^\omega$ admits a continuous injection into $C_p(\omega_1+1, \{0,1\})^\omega$. Since $C_p(X,Y)^\omega$ is homeomorphic to $C_p(X, Y^\omega)$ (see \cite[Proposition 0.3.3]{ARH}), we conclude
that $C_p(\omega_1, \{0,1\}^\omega)$ admits a continuous injection into $C_p(\omega_1+1, \{0,1\}^\omega)$. However,
$C_p(\omega_1, \{0,1\}^\omega)$ does not embed into $C_p(\omega_1+1, \{0,1\}^\omega)$ since the latter has countable tightness
while the former does not. Nonetheless, we believe that the mentioned earlier result may have a chance to be extended to the class
of first-countable countably compact subspaces of ordinals.

\par\bigskip\noindent
\begin{que}
Let $X$ be a countably compact first-countable subspace of an ordinal and 
 $Z$ a subspace of $X$. Is it true that $C_p(Z)$ continuously injects into $C_p(X)$ iff
$C_p(Z)$ embeds into $C_p(X)$? Is it true that $C_p^\omega(Z)$ continuously injects into $C_p(X)^\omega$ iff
$C_p(Z)^\omega$ embeds into $C_p(X)$?
\end{que}

\par\bigskip\noindent
{\bf Acknowledgment.} {\it
The author would like to thank the referee for many helpful remarks, corrections, and suggestions.
}

\end{document}